\documentclass[a4paper,11pt]{article}
\usepackage[english]{babel}
\usepackage[utf8]{inputenc}
\usepackage{amsmath,amsthm,amssymb}
\usepackage{enumerate}
\usepackage{hyperref}
\usepackage{xcolor,graphicx}

\numberwithin{equation}{section}
\usepackage{tikz}
\usepackage[margin=1.2in]{geometry}

  \newcommand{\CC}{{\mathcal C}}
  \newcommand{\DD}{{\mathcal D}}

  \newcommand{\sfT}{{\sf T}}
  \newcommand{\bDelta}{{\mathbf \Delta}}
  \newcommand{\tDelta}{{\widetilde \Delta}}
  
  \newcommand{\mO}{{\mathcal O}}
  
  \newcommand{\R}{\mathbb{R}} 
  \newcommand{\C}{\mathbb{C}} 
  \newcommand{\N}{\mathbb{N}} 
  \newcommand{\Z}{\mathbb{Z}} 
  
  \DeclareMathOperator{\Tr}{\mathsf{Tr}}
 
\newcommand{\beq}{\begin{equation}}
\newcommand{\eeq}{\end{equation}}
\newcommand{\beqn}{\begin{equation*}}
\newcommand{\eeqn}{\end{equation*}}
\newcommand{\beqr}{\begin{eqnarray}}
\newcommand{\eeqr}{\end{eqnarray}}
\newcommand{\beqrn}{\begin{eqnarray*}}
\newcommand{\eeqrn}{\end{eqnarray*}}
\newcommand{\bmline}{\begin{multline}}
\newcommand{\emline}{\end{multline}}
\newcommand{\bmlinen}{\begin{multline*}}
\newcommand{\emlinen}{\end{multline*}}

\newtheorem{defin}{Definition}[section]
\newtheorem{definition}[defin]{Definition}
\newtheorem{proposition}[defin]{Proposition}
\newtheorem{theorem}[defin]{Theorem}
\newtheorem{rem}[defin]{Remark}
\newtheorem{corollary}[defin]{Corollary}
\newtheorem{lemma}[defin]{Lemma}

\newcommand{\cgpsd}{\mathcal{CS}_{+}}
\newcommand{\cpsd}{\mathcal{S}_{+}}

\newcommand{\cp}{\mathcal{CP}}

\newcommand{\ignore}[1]{}

\newcommand{\SSS}{{\mathcal S}}

\newcommand{\cop}{{{\mathcal {COP}}}}

\newcommand{\cl}{\text{\rm cl}}

\newcommand{\uX}{\underline{X}}

\newcommand{\oDelta}{\underline \bDelta}
\newcommand{\tchiq}{\widetilde\chi_q}
\usepackage{authblk}

\newcommand{\csvn}{\mathcal{CS}_{\rm{vN}+}}
\newcommand{\csu}{\mathcal{CS}_{\cU+}}
\newcommand{\cF}{\mathcal F}
\newcommand{\pow}{\mathcal P}
\newcommand{\cU}{\mathcal U}
\newcommand{\cR}{\mathcal R}
 
\newcommand{\ep}{\varepsilon}
 \DeclareMathOperator{\nt}{\mathsf{tr}}
  
\newcommand{\cM}{\mathcal M}
\newcommand{\cH}{\mathcal H}
\newcommand{\cN}{\mathcal N}
\newcommand{\cB}{\mathcal B}

\title{On the closure of the completely positive semidefinite cone and linear approximations to quantum colorings}

\author[1]{Sabine Burgdorf\thanks{burgdorf@cwi.nl}}
\author[1,2]{Monique Laurent\thanks{monique@cwi.nl}}
\author[1]{Teresa Piovesan\thanks{piovesan@cwi.nl}}

\affil[1]{Centrum Wiskunde \& Informatica (CWI)\\ Amsterdam, The Netherlands}
\affil[2]{Tilburg University \\ Tilburg, The Netherlands}

\begin{document}
\maketitle

\begin{abstract}
We  investigate structural properties  of the completely positive semidefinite cone $\cgpsd^{n}$, consisting of all the $n \times n$ symmetric matrices that admit a Gram representation by positive semidefinite matrices of any size.
This cone has been  introduced to  model quantum graph parameters as conic optimization problems. Recently it has also been used to characterize the set 
$\mathcal Q$ of bipartite quantum correlations, as projection of an affine section of it.
We have two main results concerning  the structure of the completely positive semidefinite cone, namely about  its interior and  about its closure.
 On the one hand  we construct a hierarchy of polyhedral cones which covers the interior of $\cgpsd^n$, which we use for computing some variants of the quantum chromatic number by way of a linear program.
 On the other hand we  give an explicit description of the closure of the completely positive semidefinite cone, by showing that it consists of all matrices admitting a Gram representation in the tracial ultraproduct  of matrix algebras. 
\end{abstract}

\section{Introduction}\label{sec:intro}

\subsection*{General background}
Entanglement, one of the most peculiar features of quantum mechanics,  allows different parties to be correlated in a non-classical way. 
Properties of entanglement can be studied through the set of bipartite quantum correlations, commonly denoted as $\mathcal Q$, consisting of  the conditional probabilities that two physically separated parties can generate by
performing measurements on a shared entangled state. 
More formally, a conditional probability distribution $(P(a,b | x,y))_{a\in A,b\in B,x\in X,y\in Y}$ is called
\emph{quantum} if $P(a,b | x,y) = \psi^{\dagger} E_{x}^{a} \otimes F_{y}^{b} \psi$ for some unit vector $\psi$ in a finite dimensional Hilbert space $\mathcal{H}$ and some sets of positive semidefinite matrices (aka measurement operators)  $\{E_{x}^{a} : a \in A\}$ and $\{ F_{y}^{b} : b \in B\}$ satisfying $\sum_{a\in A} E_{x}^{a} = I$ and $\sum_{b\in B} F_{y}^{b}=I$ for all $x \in {X}, y \in Y$.   
Clearly, we can equivalently assume that the unit vector $\psi$ is real valued and that $E_{x}^{a}, F_{y}^{b}$ are real valued positive symmetric operators. We will assume this throughout the paper.
Here we consider the case of two parties (aka the bipartite setting) and the sets $X,Y$ (resp., $A,B$) model the possible inputs (resp., outputs) of the two parties, assumed throughout to be finite.
While the set of classical correlations (those obtained using only local and shared randomness) forms a polytope so that membership can be decided using linear programming, 
 the set $\mathcal Q$ of quantum correlations is convex but with infinitely many 
 extreme points and its structure is much harder to characterize. 
An open question in quantum information is whether allowing an infinite amount of entanglement, i.e., allowing the Hilbert space $\mathcal{H}$ in the above definition to be infinite dimensional, gives rise to a probability distribution $P$ which is not quantum \cite{Wehner:2008}. In other words, it is not known whether  the set of quantum correlations $\mathcal{Q}$ is closed.
 
\smallskip
A setting which is frequently used to study the power of quantum correlations is the one of \emph{nonlocal games}. In a nonlocal game a referee gives to each of the two cooperating players a question 
and, without communication throughout the game, they have to answer. According to some known predicate, which depends on the two questions and on the two answers, the referee determines whether the players have won or lost the game. In a quantum strategy the players can use quantum correlations to answer. 
The \emph{quantum coloring game} is a particular nonlocal game that  has received a substantial amount of attention lately \cite{Avis:2006,Cameron:2007,Roberson:2012,Paulsen:2013,Ji:2013,LP13,Paulsen:2014}.
Here, each of the two players receives a vertex of a fixed graph $G$. They win if they output the same color upon receiving the same vertex or if they output different colors on pairs of adjacent vertices.
The \emph{quantum chromatic number} $\chi_{q}(G)$ is the minimum number of colors that the players must use as output set in order to win the coloring game on all input pairs with a quantum strategy.
It is not hard to see that if the players are restricted to classical strategies then
the minimum number of colors they need to win the game on all input pairs is exactly the classical chromatic number $\chi(G)$.

\smallskip
Like its classical analog the quantum chromatic number is an NP-hard graph parameter~\cite{Ji:2013}. Moreover, it is also lower bounded by the  theta number \cite{Roberson:2012}, 
which can be efficiently computed with semidefinite programming. However,  it appears to be hard to find non-trivial improved upper and lower bounds to $\chi_q(G)$. 
With the intention of better understanding $\chi_{q}(G)$ and other related quantum graph parameters, two of the authors have introduced the 
{\em completely positive semidefinite cone}  $\cgpsd^{n}$~\cite{LP13}.

Throughout $\mathcal S^n$ is the set of  real symmetric $n\times n$ matrices and $\mathcal S^n_+$ the subset of positive semidefinite matrices; 
$\langle X,Y \rangle =\Tr(XY)$ is the trace inner product and $\Tr(X)=\sum_{i=1}^n X_{ii}$ for $X,Y\in\mathcal S^n$. 
Then, $\cgpsd^n$ consists  of all matrices $A$ that admit a Gram representation by positive semidefinite matrices, i.e., such that $A=(\langle X_i,X_j\rangle)_{i,j=1}^n$ for some matrices $X_1,\ldots,X_n\in \cpsd^d$ and $d\ge 1$. 
(When we do not want to specify the size of the matrices in $\cgpsd^n$ we omit the superscript and write $\cgpsd$.)
Using an equivalent formulation of the quantum chromatic number proven in~\cite{Cameron:2007}, it is shown in \cite{LP13} that the parameter $\chi_q(G)$ can be rewritten 
as a feasibility program over the completely positive semidefinite cone:
\beq \label{def:chiq}
\chi_{q}(G) = \min t \in \mathbb{N} \text{ s.t. } \exists A \in \cgpsd^{nt}, \, A \in \mathcal{A}^{t}  \text{ and }  L_{G,t}(A) = 0.
\eeq
Here, $n$ is fixed and equal to the number of vertices of the graph $G$ while $t$ is the variable that triggers the size of the matrix variable  $A$ in the above program.
Indeed,  $A$ is indexed by $V(G)\times [t]$.
With $\mathcal{A}^t$ we represent the affine space in $\mathcal S^{nt}$ defined by the equations 
\begin{equation}\label{def:At}
\sum_{i,j \in [t]} A_{ui,vj} = 1 \text{ for } u,v \in V(G),
\end{equation}
and with $L_{G,t}: \SSS^{nt} \to \mathbb{R}$ we denote the linear map defined by
\begin{equation}\label{def:Lt}
L_{G,t}(A) = \sum_{u\in V(G),i \neq j\in [t]} A_{ui,uj} + \sum_{uv \in E(G),i\in [t]} A_{ui,vi}.
\end{equation}

Notice that any matrix in $\cgpsd$ is positive semidefinite. Moreover it has nonnegative entries because the inner product of two positive semidefinite matrices is nonnegative. Hence the condition 
$L_{G,t}(A) = 0$ is equivalent to requiring  that all the terms  in the sum in (\ref{def:Lt}) are equal to zero.
The constraint $A \in \mathcal{A}^{t}$ models that the players are using a conditional probability distribution for their strategy, while $L_{G,t}(A) = 0$ imposes that they have a winning strategy for the coloring game. 
The structure of the matrix cone $\cgpsd$ is still largely unknown. In particular it is
not known whether the cone $\cgpsd$ is a closed set. 

By replacing in (\ref{def:chiq}) the cone $\cgpsd$ by its closure $\cl(\cgpsd)$, we get another graph parameter, denoted as $\widetilde \chi_q(G)$.
Namely,
\beq  \label{def:tchiq}
\tchiq(G) = \min t \in \mathbb{N} \text{ s.t. } \exists A \in \cl(\cgpsd^{nt}), \, A \in      \mathcal{A}^{t}  \text{ and }  L_{G,t}(A) = 0.
\eeq
Clearly, $\widetilde \chi_q(G)\le \chi_q(G)$, with equality if $\cgpsd$ is closed. This parameter, which was introduced in \cite{LP13}, will be studied in this paper.

\smallskip
Interestingly, Man\v{c}inska and Roberson~\cite{MR:2014} showed recently 
that the set $\mathcal Q$ of quantum bipartite correlations can also be described in terms of the completely positive semidefinite cone. 
They show that $\mathcal Q$ 
can be obtained as  the projection of an affine section of the completely positive semidefinite cone.

\begin{theorem} \cite{MR:2014}\label{thm:corr}
A conditional probability distribution $P = (P(a,b|x,y))$ with input sets $X,Y$ and output sets $A,B$ is quantum (i.e., $P \in \mathcal Q$) if and only if there exists a matrix $R \in \cgpsd$ indexed by $(X \times A) \cup (Y\times B)$ satisfying the conditions:
\begin{align}
& \;\sum_{a,a' \in A } R_{xa,x'a'} = 1 \text{ for all } x,x'\in X,
\label{affine1} \\
& \,\,\sum_{b,b'\in B } R_{yb,y'b'} = 1 \text{ for all } y,y'\in Y,
\label{affine1b}\\
& \sum_{a\in A,b\in B}R_{xa,yb}=1 \text{ for all } x\in X, y\in Y,
\label{affine1c}\\
& \,R_{xa,yb}=P(a,b|x,y) \text{ for all } a \in A, b\in B, x \in X, y \in Y. \label{project}
\end{align}
In other words, 
\beq\mathcal Q= \pi(\cgpsd^{N}\cap \mathcal{B}^t),\eeq
where $N = |(X \times A) \cup (Y\times B)|$,  $\mathcal{B}^t$ is the affine space defined by the constraints (\ref{affine1}), (\ref{affine1b}) and 
(\ref{affine1c}),  
and $\pi$ is the projection onto the subspace indexed by $(X\times A) \times  (Y\times B)$ (defined by (\ref{project})). 
\end{theorem}
Notice that any feasible matrix $R$ to the above program has the form $\left(\begin{smallmatrix} R_{1} & P \\ P^{T} & R_{2} \end{smallmatrix}\right)$, where $R_{1}$ is indexed by $X \times A$, $R_{2}$ is indexed by $Y \times B$ and each entry of $P$ is such that  $P_{xa,yb} = P(a,b|x,y)$.

As shown  in~\cite{MR:2014},
if the completely positive semidefinite cone is closed then the set $\mathcal Q$ of quantum bipartite correlations too is closed.
Indeed,  the constraints (\ref{affine1})-(\ref{affine1c}) imply that the set $\cgpsd\cap \mathcal B^t$ is bounded. Hence, if $\cgpsd $ is closed then  
 $\cgpsd\cap \mathcal B^t$ is compact and thus its
projection $\mathcal Q=\pi(\cgpsd\cap \mathcal B^t)$ is compact.

\subsection*{Our contributions}
The results of this paper are twofold. First we construct a hierarchy of polyhedral cones that asymptotically covers the interior of the completely positive semidefinite cone 
$\cgpsd$. Moreover we show how this hierarchy can be used to study the quantum chromatic number. In particular we build a hierarchy of linear programs, among which one of them permits to compute the variant $\widetilde \chi_q(G)$ in (\ref{def:tchiq}) 
of the parameter $\chi_q(G)$.
This idea can also be applied to compute variants of other versions of the quantum chromatic number; we will indicate how to do that for the variant $\widetilde \chi_{qa}(G)$ of the  parameter $\chi_{qa}(G)$ considered in \cite{Paulsen:2014}. 
See below for some details and Sections~\ref{sec:poly} and~\ref{sec:lp-bounds} for the proofs.

As a second main contribution we provide an explicit description of the
closure of the cone $\cgpsd$, in terms of tracial ultraproducts of
matrix algebras. Moreover we exhibit a larger cone, containing
$\cgpsd$,  which can be interpreted as an infinite dimensional analog of
$\cgpsd$. This cone consists of the matrices
which admit a Gram representation by (a specific class of) positive
semidefinite operators on a possibly infinite dimensional Hilbert space
instead of Gram representations by {\em finite} positive semidefinite
matrices.
We can in fact show that this larger cone is indeed a closed cone and that
it is equal to $\cl(\cgpsd)$ if Connes' embedding conjecture holds true. Since
the description of these cones involve quite some notation and concepts
from operator theory we skip a preliminary description of the used
methods and refer directly to Section \ref{sec:closure} which can be
read independently of the other part.

In summary, our results give structural information about the completely positive semidefinite cone $\cgpsd$ which come in two flavors, depending whether we consider its interior or its boundary.

\medskip
We now give some more details about our first contribution.
In a nutshell, the idea for building  the hierarchy of polyhedral cones is to discretize the set of positive semidefinite matrices by rational ones with bounded entries.
Namely, given an integer $r\ge 1$, we define the cone $\CC^n_r$ as the conic hull of all matrices $A$ that admit a Gram representation by $r\times r$ positive semidefinite matrices $X_1,\ldots,X_n$ whose entries are rational with denominator at most $r$ and satisfy $\sum_{i=1}^n\Tr(X_i)=1$.
We show that the cones $\CC^n_r$ and their dual cones $\DD_r^{n} = \CC_{r}^{n*}$ satisfy the following properties:
\beqn\text{\rm int}(\cgpsd^n)\subseteq \bigcup_{r\ge 1}\CC^n_r \subseteq \cgpsd^n \ \text{ and } \
\cgpsd^{n*}= \bigcap_{r\ge 1} \DD_r^{n}. \eeqn
Moreover, for any fixed $r$,  linear optimization over the cone $\CC^n_r$ can be performed in polynomial time in terms of $n$.
This discretization idea was also used in the classical (scalar) setting, where a hierarchy of polyhedral cones is constructed to approximate the completely      positive cone 
(consisting of all matrices that admit a Gram representation by nonnegative vectors) 
and its dual, the copositive cone (see \cite{Yil}). Our construction is in fact inspired by this classical counterpart.
Discretization is also widely used in optimization to build good approximations for polynomial optimization problems over the standard simplex or for evaluating tensor norms (see e.g. \cite{BdK}, \cite{dKLP},  the recent work \cite{BH} and references therein).

\smallskip
One of the difficulties in using the cone $\cgpsd$ for studying the quantum parameter $\chi_q(G)$ or general quantum correlations  in $\mathcal Q$ stems from the fact that the additional affine conditions posed on the matrix $A\in \cgpsd$ imply that it must lie on the boundary of the cone $\cgpsd$.
This is the case for instance for the conditions that $A$ must belong to the affine space $\mathcal A^t$ in (\ref{def:At}),
or the condition $L_{G,t}(A)=0$ in (\ref{def:Lt}), or the conditions (\ref{affine1}), (\ref{affine1b}) and (\ref{affine1c}).
Since we do not know whether the cone $\cgpsd$ is closed, this is why we may get different parameters depending whether we use the cone $\cgpsd$ or its closure.  

In order to be able to exploit the fact that the cones 
 $\CC^{n}_{r}$ asymptotically cover the full interior of $\cgpsd^{n}$, 
we will relax the affine constraints (using a small perturbation) to ensure the existence of a feasible solution in the interior of the cone $\cgpsd$. In this way we will 
be able to get a hierarchy of parameters that can be computed through linear programming and give the exact value of $\widetilde \chi_q(G)$.
We remark that this result is existential, we can prove the existence of a linear program permitting to compute the quantum parameter  but we do not know at which stage this happens. 
This result should be seen in the light of a recent result of the same flavor proved in 
\cite{Paulsen:2014}. 
The authors of \cite{Paulsen:2014} consider yet another variant $\chi_{qc}(G)$ of the quantum parameter $\chi_q(G)$, satisfying $\chi_{qc}(G)\le \chi_q(G)$,  and they
show that $\chi_{qc}(G)$ can be computed with a positive semidefinite program (also not explicitly known).
 The definition of $\chi_{qc}(G)$ is given below.

\subsection*{Link to other variants of the quantum chromatic number}
In the papers  \cite{Paulsen:2013, Paulsen:2014}, Paulsen and coauthors have introduced many variants of the quantum chromatic number motivated by the study of quantum correlations. We recall two of them, the parameters $\chi_{qa}(G)$ and $\chi_{qc}(G)$, in order to pinpoint the link to our parameter $\widetilde \chi_q(G)$ and to our approach.

Recall that the quantum chromatic number $\chi_q(G)$ is the minimum number of colors that the  players must use to always win 
the corresponding coloring game with a quantum strategy.
In other words, this is the minimum integer $t$ for which there exists a probability $P =(P(i,j|u,v))\in \mathcal{Q}$ with input sets $X= Y = V(G)$ and output sets $A=B=[t]$, such that $P(i,j | u,u) = 0$ for all $i \neq j \in [t]$ and $u \in V(G)$, and $P(i,i | u,v) = 0$ for all $i \in [t]$ and $uv\in E(G)$.
For convenience, in the following paragraphs we will omit the dependence of $P$ on $t$, which should be considered as implicit.
Forcing the probability of these combinations of inputs and output to be zero imposes that the players have a winning strategy. We combine those constraints into a single one by defining the linear map $\mathcal L_{G,t}:\R^{(nt)^{2}}\rightarrow \R$ by 
\beqn\mathcal{L}_{G,t}(P) = \sum_{i \neq j \in [t], u \in V(G)} P(i,j | u,u) + \sum_{i \in [t],uv\in E(G)} P(i,i | u,v).\eeqn
Then, the players have a winning strategy if and only if the probability $P$ satisfies $\mathcal{L}_{G,t}(P)= 0$.
The following is the original definition of $\chi_q(G)$ in \cite{Cameron:2007}:
\beqn\chi_{q}(G) = \min t \in \mathbb{N} \text{ s.t. } \exists P \in \mathcal{Q} \text{ with } \mathcal{L}_{G,t}(P)= 0.\eeqn
In \cite{Cameron:2007} it is shown that 
 in the coloring game the optimal quantum strategy is symmetric: the two players perform the same action upon receiving the same input.
This special additional structure of the coloring game is the reason why $\chi_{q}(G)$ can be equivalently reformulated as in (\ref{def:chiq}).

The parameter $\chi_{qa}(G)$ defined in~\cite{Paulsen:2013} asks the probability $P$ to be in the closure of $\mathcal{Q}$:
\beqn \label{pr:chiqa}
\chi_{qa}(G) = \min t \in  \mathbb{N} \text{ s.t. } \exists P \in \cl(\mathcal{Q}) \text{ with } \mathcal{L}_{G,t}(P) = 0.
\eeqn
Hence, the following relationship holds: $\chi_{qa}(G)\le \chi_q(G)$.

The authors of \cite{Paulsen:2013} (see also  \cite{Paulsen:2014}) also consider probability distributions arising from the relativistic point of view.
Roughly, instead of assuming that the measurement operators act on different Hilbert spaces so that  joint measurements have a tensor product structure, in the relativistic model 
the measurement operators act on a common Hilbert space and the operators of the two parties commute mutually.
 In this case, joint measurement operators have a product structure. 
More formally, a correlation $P=(P(a,b| x,y))$ is obtained from relativistic quantum field theory if it is of the form 
$P(a,b | x,y) = \psi^{\dagger} E_{x}^{a} F_{y}^{b} \psi$,
where $\psi$ is a unit vector in a (possibly infinite dimensional) Hilbert space $\mathcal{H}$, $E_{x}^{a}$ and $F_{y}^{b}$ are positive operators on $\mathcal{H}$ satisfying $\sum_{a \in A} E_{x}^{a} = I = \sum_{b \in B} F_{y}^{b}$ for all $x \in X, y \in Y$ and $E_{x}^{a} F_{y}^{b} = F_{y}^{b}E_{x}^{a} $ for all $a \in A, b\in B, x \in X, y \in Y$.
We denote by $\mathcal{Q}_{c}$ the set of quantum bipartite correlations arising from the relativistic point of view.
The set $\mathcal Q_c$ is closed (see e.g.~\cite[Proposition 3.4]{Fritz:2012}) and the following inclusions hold:   
\begin{equation}\label{eq:QQ}
\mathcal Q\subseteq \cl(\mathcal Q)\subseteq \mathcal Q_c.
\end{equation}
 Deciding whether equality $\mathcal Q_c=\cl(\mathcal Q)$ holds is known to be equivalent to Connes' embedding conjecture (see \cite{Oz:2013,Fritz:2012,Junge:2011}) and deciding whether 
$\mathcal Q_c=\mathcal Q$ is known as Tsirelson's problem. 

In \cite{Paulsen:2013} the parameter $\chi_{qc}(G)$ is defined as
\beqn \label{pr:chiqc}
\chi_{qc}(G) = \min t \in  \mathbb{N} \text{ s.t. } \exists P \in \mathcal{Q}_{c} \text{ with } \mathcal{L}_{G,t}(P) = 0.
\eeqn
In~\cite{Paulsen:2014}  it is shown that $\chi_{qc}(G)$ can be computed by
a positive semidefinite program (after rounding). This result is existential in the sense that the semidefinite program is not explicitly known.
For this the authors of \cite{Paulsen:2014}  use 
the semidefinite programming hierarchy developed by Navascu\'es, Pironio and Ac\'in \cite{Navascues:2008}  for noncommutative polynomial optimization.
This technique  can be applied since the definition of $\chi_{qc}(G)$ is in terms of products of operators. 
Note  that this technique cannot be applied to the parameters $\chi_{qa}(G)$ and $\chi_q(G)$ whose definitions involve tensor products of operators. 
It is not know whether the parameters $\chi_{qa}(G)$ and $\chi_{q}(G)$ can be written as  semidefinite programs.
As pointed out in~\cite{Paulsen:2014}, 
in view of the inclusions in (\ref{eq:QQ}), the following relationships hold between the parameters:
 \beqn\chi_{qc}(G) \le \chi_{qa}(G) \le \chi_{q}(G).\eeqn

Using Theorem~\ref{thm:corr}, we can reformulate the parameters $\chi_{q}(G)$ and $\chi_{qa}(G)$ as feasibility problems over affine sections of the cones  $\cgpsd$ and $\cl(\cgpsd)$, respectively.
Namely, we have 
\begin{align*}
\chi_{q}(G)&= \min t  \text{ s.t. } \exists P \in \pi(\cgpsd^{2nt} \cap \mathcal{B}^t) \text{ with } \mathcal{L}_{G,t}(P) = 0, \ \text{ and}\\
\chi_{qa}(G)&= \min t  \text{ s.t. } \exists P \in \cl(\pi(\cgpsd^{2nt} \cap \mathcal{B}^t)) \text{ with } \mathcal{L}_{G,t}(P) = 0.
\end{align*}
Recall that we introduced the variant $\widetilde \chi_q(G)$ by replacing the cone $\cgpsd$ by its closure in the definition (\ref{def:chiq}) of $\chi_q(G)$.
Analogously, we introduce the variant $\widetilde \chi_{qa}(G)$ by replacing $\cgpsd$ by its closure in the above definition of $\chi_{qa}(G)$. Namely, 
\begin{equation}\label{def:tchiqa}
\widetilde{\chi}_{qa}(G) = \min t  \text{ s.t. } \exists P \in \pi(\cl(\cgpsd^{2nt}) \cap \mathcal{B}^t) \text{ with } \mathcal{L}_{G,t}(P) = 0.
\end{equation}
Note that the set $\cl(\cgpsd)\cap\mathcal B^t$ is bounded, thus compact, so that its projection
 $\pi (\cl(\cgpsd) \cap \mathcal{B}^t)$ is compact too.
 Hence the inclusion $\cgpsd \cap \mathcal{B}^t \subseteq \cl(\cgpsd) \cap \mathcal{B}^t$ implies:
\beqn\cl( \pi (\cgpsd \cap \mathcal{B}^t)) \subseteq \pi (\cl(\cgpsd) \cap \mathcal{B}^t)\eeqn 
and thus 
the following relationship: $\widetilde{\chi}_{qa}(G)\le \chi_{qa}(G)$.
In Section~\ref{sec:lp-bounds} we will show that $\widetilde{\chi}_{qa}$ can be computed with a linear program.

Moreover, note that if  a matrix $A$ is feasible for   the program (\ref{def:tchiq}) defining $\tchiq(G)$, then
the matrix $R = \left(\begin{smallmatrix} A & A \\ A & A \end{smallmatrix}\right) $ is feasible for the program
(\ref{def:tchiqa}) defining $\widetilde\chi_{qa}(G)$. Hence, $\widetilde{\chi}_{qa}(G) \le \tchiq(G)$ holds.

\medskip
The relationship between the parameters $\chi_{q}(G), \chi_{qc}(G), \chi_{qa}(G)$ and $\widetilde{\chi}_{qa}(G), \tchiq(G)$ can be summarized as follows: 
%
\begin{center}
\begin{tikzpicture} 
\tikzstyle{every node}=[font=\small]
\coordinate [label=center:$\chi_{qc}(G)$] (qc) at (0.8, 3);
\coordinate [label=center:$\le$] (l1) at (1.7, 3);
\coordinate [label=center:$\chi_{qa}(G)$] (qa) at (2.6, 3);
\coordinate [label=center:$\le$] (l2) at (3.5, 3);
\coordinate [label=center:$\chi_{q}(G)$] (q) at (4.4, 3);

\node at (2.2, 2.45) [rotate=70] {$\le$};
\node at (4.0, 2.45) [rotate=70] {$\le$};

\coordinate [label=center:$\widetilde{\chi}_{qa}(G)$] (tqa) at (2.0, 1.95);
\coordinate [label=center:$\le$] (l3) at (2.9, 1.95);
\coordinate [label=center:$\tchiq(G)$] (tq) at (3.8, 1.95);
\end{tikzpicture}
\end{center}

\section{Polyhedral approximations of \texorpdfstring{$\cgpsd$}{CS} and its dual cone \texorpdfstring{$\cgpsd^*$}{CSd}}\label{sec:poly}

In this section we construct hierarchies of polyhedral cones converging asymptotically to the completely positive cone and its dual. 
We start in Section~\ref{sec:cpsd} by recalling the definition of $\cgpsd$ and of $\cgpsd^{*}$ as well as some useful properties and introduce the new hierarchy in Section~\ref{sec:approx-cpsd}. 
The construction of our polyhedral hierarchy 
is directly inspired from  the classical case where analogous hierarchies of polyhedral cones 
exist for approximating the completely positive cone $\cp^n$ and  the copositive cone $\cop^n$; in Section~\ref{sec:approx-cop} we recall this construction.

\subsection{The completely positive semidefinite cone and its dual}\label{sec:cpsd}

The completely positive semidefinite cone was introduced in \cite{LP13} to study graph parameters arising from quantum nonlocal games and quantum information theory. It has also been considered implicitly in \cite{FW:2014}.

Recall that a matrix $A \in \SSS^{n}$ is positive semidefinite if and only if it admits a Gram representation by vectors, i.e., if
$A=(\langle x_i,x_j\rangle)_{i,j=1}^n$ for some $x_1,\ldots,x_n\in \R^d$ and $d\ge 1$.
 We write $A\succeq 0$ (resp., $A\succ 0$) when $A$ is positive semidefinite (resp., positive definite) and $\mathcal S^n_{+}$ is the set of positive semidefinite matrices.

\begin{definition}
The completely positive semidefinite cone $\cgpsd^{n}$ is the set of symmetric matrices $A$ which admit a Gram representation by positive semidefinite matrices, i.e., $A  = (\langle X_{i}, X_{j} \rangle)_{i,j}$ for some $X_{1},\dots,X_{n} \in \SSS_{+}^{d}$ and $d \in \mathbb{N}$.
\end{definition}

The completely positive cone $\cp^{n}$ is the set of symmetric matrices that admit a Gram representation by nonnegative vectors:
$A \in \cp^{n}$ if $A =  (\langle x_{i}, x_{j}\rangle)_{i,j}$ for some $x_{1},\dots,x_{n} \in \mathbb{R}_{+}^{d}$ and $d \in \mathbb{N}$.
Hence $\cp^n$ can be considered as the classical analog  of $\cgpsd^{n}$.
Clearly every  completely  positive semidefinite matrix is positive semidefinite and nonnegative, and every completely positive matrix is completely positive semidefinite. That is, we have the following relationships between these cones:
\beqn\cp^{n} \subseteq \cgpsd^{n} \subseteq \SSS_{+}^{n} \cap \mathbb{R}^{n \times n}_{+}.\eeqn
In~\cite{LP13} it is shown that all these inclusions are strict for $n \ge 5$ (see also \cite{FW:2014}).
 For $n\le 4$ it is well known that $\cp^{n} = \SSS_{+}^{n} \cap \mathbb{R}^{n \times n}_{+}$.
For this and other properties of $\cp$ we refer the reader to the book~\cite{Berman:2003}.
Both $\cp^{n}$ and $\SSS_{+}^{n}$ are closed cones, while we do not know whether $\cgpsd^{n}$ is closed. 

\smallskip
Moving on to the dual side, as noted in \cite{LP13}, the dual cone of $\cgpsd^{n}$ has a simple characterization in terms of trace nonnegative polynomials. 
Given a matrix $M\in \SSS^n$, define the polynomial $p_M=\sum_{i,j=1}^n M_{ij} x_ix_j$ in $n$ noncommutative variables.  Then $M$ belongs to the dual cone $\cgpsd^{n*}$ precisely when $\Tr(p_M(X_1,\ldots,X_n))\ge 0$ for all $n$-tuples   $\uX=(X_1,\ldots,X_n) \in \cup_{d\ge 1} (\cpsd^d)^n$.
If we require nonnegativity only for all $\uX \in \R^n_+$ (i.e., the case $d=1$), which amounts to requiring that the polynomial $p_M$ takes nonnegative values when evaluated at any point in $\R^n_+$, then the matrix $M$ is said to be {\em copositive};  $\cop^{n}$ denotes the cone of copositive matrices.
The cones $\cp^n$ and $\cop^n$ are dual to each other:  $\cop^{n} = \cp^{n*}$ and, 
 by duality,
we have the inclusions:
\beqn\SSS_{+}^{n} +  (\SSS^{n} \cap  \mathbb{R}^{n \times n}_{+} ) \subseteq \cgpsd^{n*} \subseteq \cop^{n}.\eeqn

\smallskip
As will be explained in detail in Section \ref{sec:lp-bounds}, in order to be able to use our polyhedral hierarchy, we will need to have matrices that are in the interior of $\cgpsd$. Recall that a matrix $A\in \cgpsd$ lies in the interior of $\cgpsd$ if and only if $\langle A,M\rangle >0$ for all nonzero matrices $M\in \cgpsd^{*}$.
Hence,  $A$ lies in the boundary of $\cgpsd$ if and only if there exists a nonzero matrix $M \in \cgpsd^{*}$ such that $\langle A,M \rangle = 0$.
For further reference we observe that matrices in $\cgpsd$ with a zero entry, or lying in the affine spaces $\mathcal A^t$ or  $\mathcal B^t$,
lie in the boundary of $\cgpsd$.
\begin{lemma}\label{lemborder}
Consider a matrix $A$ in the cone $\cgpsd$ (of appropriate size).
 Then $A$ lies in the boundary of $\cgpsd$ in any of the following cases: (i) $A$ has a zero entry;
(ii) $A$ belongs to the affine space $\mathcal A^t$ defined by (\ref{def:At}), or (iii) $A$ belongs to the affine space $\mathcal B^t$ defined by the conditions 
(\ref{affine1}), (\ref{affine1b}) and (\ref{affine1c}).
\end{lemma}
\begin{proof}
$(i)$ Say $A$ has a zero entry: $A_{ij}=0$. Then $\langle A,E_{ij}\rangle =0$, where $E_{ij}$ is the elementary matrix (with all zero entries except entry 1 at positions $(i,j)$ and $(j,i)$).
As $E_{ij}$ is nonnegative it belongs to $ \cgpsd^{n*}$, and thus  $A$ lies in the boundary of $\cgpsd^n$.

$(ii)$ Assume now that $A\in\cgpsd^{nt}$ lies in $\mathcal A^t$. Pick two distinct nodes $u,v\in V(G)$ and consider the matrix $M=J\otimes F$, 
where $J$ is the $t \times t$ all-ones matrix and $F$ is the $n \times n$  matrix with $F_{uu} = F_{vv} = 1$, $F_{uv} = F_{vu} = -1$ and zero elsewhere. Then, $M\succeq 0$ since $J,F\succeq 0$ and thus $M\in \cgpsd^{nt*}$. Moreover,
$\langle A,M\rangle =0$, showing that $A$ lies on the boundary of $\cgpsd^{nt*}$.

Same argument  in case $(iii)$.
\end{proof}

\subsection{Polyhedral approximations of \texorpdfstring{$\cp^n$}{CP} and \texorpdfstring{$\cop^n$}{COP}}\label{sec:approx-cop}

As mentioned above, the copositive cone $\cop^n$ consists of all matrices $M\in \SSS^n$  
for which the polynomial $p_M=\sum_{i,j=1}^n M_{ij}x_ix_j$ is nonnegative over $\R^n_+$. 
Alternatively, a matrix $M\in \SSS^n$ is copositive if and only if the polynomial $p_M$ is nonnegative over the standard simplex
\beqn\Delta_n=\{x\in \R^n_+: \sum_{i=1}^nx_i=1\}.\eeqn

The idea for constructing outer approximations of the copositive cone is simple and relies on requiring nonnegativity of the polynomial $p_M$ over  all rational points in the standard simplex with given denominator $r$ and letting $r$ grow. This idea is made explicit in \cite{Yil} and goes back to earlier work on how to design tractable approximations for quadratic optimization problems over the standard simplex \cite{BdK,dKP} and more general polynomial optimization problems \cite{dKLP}.
More precisely, for an integer $r\ge 1$, define the sets
\beqn\Delta(n,r)=\{x\in \Delta_n: rx \in \Z^n\}, \quad  \tDelta(n,r)=\bigcup_{s=1}^r \Delta(n,s)\eeqn
where we restrict to rational points in $\Delta_n$ with given denominators.
The sets $\tDelta(n,r)$ are  nested within the standard simplex: $\tDelta(n,r)\subseteq\tDelta(n,r+1)\subseteq \Delta_n$.
Now, following Yildirim \cite{Yil}, define the cone:
\beqn\mO^n_r=\{M\in\SSS^n: x^\sfT Mx\ge 0\ \forall x\in \tDelta(n,r)\},\eeqn
and its dual cone $\mO^{n*}_r$, which is the conic hull of all matrices of the form  $vv^\sfT$ for some $v\in \tDelta(n,r)$.
By construction, the cones $\mO^n_r$ form a hierarchy of outer approximations for $\cop^n$ and their dual cones form a hierarchy of inner approximations for $\cp^n$:
\beqn\cop^n\subseteq \mO^n_{r+1}\subseteq \mO^n_r \ \text{ and } \ 
\mO^{n*}_r\subseteq \mO^{n*}_{r+1}\subseteq \cp^n.\eeqn
Yildirim \cite{Yil} shows the following convergence results.

\begin{theorem}\cite{Yil}
We have: $\cop^n=\bigcap_{r\ge 1}\mO^n_r$. Moreover,   $\text{\rm int} (\cp^n)\subseteq \bigcup_{r\ge 1}\mO^{n*}_r\subseteq \cp^n$ 
and $\cp^n$ is 
equal to 
the closure of the set $\bigcup_{r\ge 1}\mO^{n*}_r$.
\end{theorem}

\subsection{The new cones \texorpdfstring{$\CC^n_r$}{Cr} and \texorpdfstring{$\DD^n_r$}{Dr}}\label{sec:approx-cpsd} 

We now introduce the cones $\CC^n_r$, which will form a hierarchy of inner approximations for the cone $\cgpsd^n$, and the cones $\DD^n_r$, which  will form a hierarchy of outer approximations for the dual cone $\cgpsd^{n*}$.
These cones are in fact dual to each other, so it suffices to define the cones $\DD^n_r$. 
The idea is simple and analogous to the idea used in the classical (scalar) case: instead of requiring trace nonnegativity of the polynomial $p_M$ over the full
set $\cup_{d\ge 1}(\SSS^d_+)^n$, we only ask trace nonnegativity over specific finite subsets. 
We start with defining the set
\begin{equation}\label{setD}
\bDelta_n =\{\uX=(X_1,\ldots,X_n)\in \bigcup_{d\ge 1} (\cpsd^d)^n: \sum_{i=1}^n\Tr(X_i)=1\},
\end{equation}
which can be seen as the dimension-free matrix analog of the standard simplex $\Delta_n$ in $\R^n$. 
As we now observe, a matrix $M$ belongs to $\cgpsd^{n*}$ if and only if its associated polynomial $p_M$ is trace nonnegative on all $n$-tuples of {\em rational} matrices in  $\bDelta_n$.
\begin{lemma}\label{lemrational}
For $M\in \SSS^n$,  the following assertions are equivalent:
\begin{enumerate}[(i)]
\item $M\in \cgpsd^{n*}$, i.e., $\Tr(p_M(\uX))\ge 0$ for all $\uX\in \cup_{d\ge 1}(\SSS^d_+)^n$.
\item
$\Tr(p_M(\uX))\ge 0$ for all $\uX \in \bDelta_n$.
\item
$\Tr(p_M(\uX))\ge 0$ for all $\uX=(X_1,\ldots,X_n) \in \bDelta_n$ with $X_1\succ 0,\ldots,X_n\succ 0$.
\item
$\Tr(p_M(\uX))\ge 0$ for all $\uX=(X_1,\ldots,X_n) \in \bDelta_n$ with $X_1\succ 0,\ldots,X_n\succ 0$ and with rational entries.
\item
$\Tr(p_M(\uX))\ge 0$ for all $\uX \in \bDelta_n$ with rational entries.
\end{enumerate}
\end{lemma}

\begin{proof}
The implications $(i) \implies  (ii)  \implies  (iii) \implies   (iv)$,
 $ (i) \implies  (v) $ and $ (v) \implies  (iv) $
are clear. We will show that $ (iv) \implies  (iii)\implies  (ii) \implies  (i)$.

Implication $(ii)\implies(i)$ follows by scaling: Let $\uX\in (\SSS^d_+)^n$ with $\lambda=\sum_{i=1}^n\Tr(X_i)>0$ (else, $\uX$ is identically zero and $\Tr(p_M(\uX))=0$). Then we have $\uX/\lambda \in \bDelta_n$ and thus $\Tr(p_M(\uX/\lambda))\ge 0$, which implies $\Tr(p_M(\uX))\ge 0$.

 The remaining implications follow using continuity arguments.
Namely, for  $(iv)\implies  (iii)$, use the fact that   the set of rational positive definite matrices is dense within the set of positive definite matrices and, for 
$(iii) \implies  (ii)$, use the fact that the set of positive definite matrices is dense within the set of positive semidefinite matrices (combined with rescaling).
\end{proof}

This motivates introducing the following subset $\bDelta(n,r)$ of the set $\bDelta_n$, obtained by considering only $n$-tuples of rational positive semidefinite matrices with denominator at most $r$.
This set can be seen as a matrix analog of the rational grid point subsets 
of the standard simplex $\Delta_n$ and it permits to define the new cones $\DD^n_r$.

\begin{definition} 
Given an integer $r\in \N$, define the set
\beqn\bDelta(n,r) = \{ \uX\in \bDelta_n: \text { each}\ X_{i} \text{ has rational entries with denominator}\leq r\} \eeqn
  and define the cone
\beqn\DD_r^n=\{M\in \SSS^n: \Tr(p_M(\uX))\ge 0 \ \forall \uX \in \bDelta(n,r)\}.\eeqn
\end{definition}

\smallskip
Next we show that the cone $\DD^n_r$ is a polyhedral cone. Indeed, as we observe below, although the set $\bDelta(n,r)$ is not finite,   we may   without loss of generality  replace in the definition of $\DD^n_r$ the set $\bDelta(n,r)$ by its subset $\oDelta(n,r)$, obtained by restricting   to $r\times r$  matrices $X_1,\ldots,X_n$.

\begin{lemma}\label{lemrational2}
Define the set
\beqn\oDelta(n,r) = \{ \uX\in (\SSS^r_+)^n : \sum_{i=1}^n \Tr(X_i)=1, \text{\small each $X_i$ has rational entries with denominator} \leq r \}.\eeqn
Then, equality holds:
\beqn\DD^n_r=\{M\in \SSS^n: \Tr(p_M(\uX))\ge 0\ \forall \uX\in \oDelta(n,r)\}.\eeqn
  \end{lemma}
\begin{proof}
The inclusion ``$\supseteq$" is clear since $\oDelta(n,r)\subseteq \bDelta(n,r)$. 

To show the reverse inclusion, 
take a matrix $M$ such that $\Tr(p_M(\uX))\ge 0$ for all $\uX \in \oDelta(n,r)$. Consider a $n$-tuple $\uX=(X_1,\ldots,X_n)\in \bDelta(n,r)$. 
The matrices $X_1,\ldots,X_n$ are rational with denominator at most $r$, $\sum_{i=1}^n\Tr(X_i)=1$ and  $X_1,\ldots,X_n\in \SSS^d_+$ with $d>r$ (else there is nothing to prove).
For each $i \in [n]$, set $ I_i=\{k\in [d]: X_i(k,k)\ne 0\}$ and notice that 
$\Tr(X_i)\ge |I_i|/r$ (since each diagonal entry $X_i(k,k)$ indexed by $k\in I_i$ is at least $1/r$).
Hence we have $1=\sum_{i=1}^n \Tr(X_i)\ge \sum_{i=1}^n | I_i|/r$, implying $\sum_{i=1}^n | I_i|\le r$.
 Then we can find a set $I$ containing $ I_1\cup \ldots \cup I_n$ with cardinality $|I|= r$.
As each matrix $X_i$ has only zero entries outside of its principal submatrix $X_i[I]$ indexed by $I$, it follows that 
$\Tr(p_M(X_1,\ldots,X_n))=\Tr(p_M(X_1[ I],\ldots, X_n[ I]))\ge 0,$ where the last inequality follows from the fact that 
$(X_1[ I],\ldots,X_n[I])$ belongs to the set $\oDelta (n,r)$.
\end{proof}

\begin{lemma}\label{lemcount}
For any fixed $r$, the cardinality of the set  $\oDelta(n,r)$ is polynomial in terms of $n$. 
More precisely, let $\gamma_r$ denote the number of $r \times r$ positive semidefinite matrices whose entries are rational with denominator at most $r$ and whose trace is at most one.
Then, $|\oDelta(n,r)| \le (\gamma_r)^r$ if $n\le r$, and $|\oDelta(n,r)|\le \binom{n}{r}(\gamma_r)^r$ if $n>r$.
\end{lemma} 

Notice that  $\Tr(p_{M}(\uX)) = \sum_{i,j} M_{ij} \langle X_{i}, X_{j} \rangle$ for any  $\uX = (X_{1},\dots,X_{n})$.
Hence, the cone $\DD^n_r$ can be equivalently defined as the set of matrices $M\in \mathcal S^n$ satisfying the (finitely many) linear inequalities:
$ \sum_{i,j=1}^n M_{ij} \langle X_{i}, X_{j} \rangle \ge 0$ for all  $(X_{1},\dots,X_{n}) \in \oDelta (n,r)$. This implies:

\begin{corollary}
The cone $\DD^n_r$ is a polyhedral cone.
\end{corollary}
As $\oDelta(n,r)\subseteq \oDelta(n,r+1)$, the sets $\DD^n_r$ form a hierarchy of outer approximations for   $\cgpsd^{n*}$:
\beqn\cgpsd^{n*}\subseteq \DD^n_{r+1}\subseteq \DD^n_r \subseteq \cdots \subseteq \DD^n_1.\eeqn
Hence, $\cgpsd^{n*}\subseteq \bigcap_{r\ge 1} \DD^n_r$. In fact, as a direct application of the equivalence of $(i)$ and $(v)$ in Lemma \ref{lemrational}, equality holds.

\begin{theorem}\label{thm0}
$\cgpsd^{n*}=\bigcap_{r\ge 1} \DD^n_r$.
\end{theorem}
We will also use the following property of the cones $\DD^n_r$.

\begin{lemma}\label{lemconv}
Consider a sequence of matrices $(M_r)_{r\ge 1}$ in $\SSS^n$  converging to a matrix $M\in \SSS^n$.
If $M_r\in \DD^n_r$ for all $r$, then $M\in \cgpsd^{n*}$.
\end{lemma}

\begin{proof}
In view of Lemma \ref{lemrational}, it suffices to show that $\Tr(p_M(\uX))\ge 0$ whenever $\uX\in \oDelta_n$ is rational valued. 
Fix a rational valued $\uX\in \oDelta_n$ and say that all its entries 
have denominator at most $t$. 
Then, $\uX \in \oDelta(n,r)$ for all $r\ge t$. 
Hence $\Tr(p_{M_r}(\uX))\ge 0$ for all $M_r$ with $r\ge t$. When $r$ tends to infinity, $\Tr(p_{M_r}(\uX))$ tends to $\Tr(p_M(\uX))$ and thus we obtain 
 that $\Tr(p_M(\uX))\ge 0$.
\end{proof}

We now turn to the description of the dual cone $\CC^n_r:= \DD^{n*}_r$. 
As a direct application of Lemma~\ref{lemrational2}, we can conclude that $\CC^n_r$ 
 is the set of conic combinations of matrices which have a Gram representation by matrices in $\oDelta(n,r)$; that is,
\beq\label{O*_r}
\CC^n_r=\text{\rm cone}\{A\in \SSS^n: A=(\langle X_i,X_j\rangle)_{i,j=1}^n \ \text{ for some } (X_1,\ldots,X_n)\in \oDelta(n,r)\}.
\eeq
By construction, the cones $\CC^n_r$ are polyhedral and they form a hierarchy of inner approximations of $\cgpsd^n$:
$\CC^n_1 \subseteq \dots \subseteq \CC^n_r \subseteq \CC^n_{r+1} \subseteq \cgpsd^n$, with strict inclusion.

\begin{lemma}
For any $n \ge 2$ and $r\ge 1$,  we have strict inclusions: $ C_{r}^{n} \subsetneq C_{r+1}^{n} \subsetneq \cgpsd^{n}$.
\end{lemma}
\begin{proof}
The only fact which needs a proof is that each inclusion is strict.
It suffices to show this for $n=2$, since one can extend a matrix $A$ in $\CC^2_r$ to a matrix in $\CC^n_r$ by adding all zero coordinates, and the same for $\cgpsd$.
For this we consider a rank 1 matrix $A=vv^T$, where $v=(1\ a)^T$ and $a$ is a nonnegative scalar. Then $A\in \cgpsd^2$.
If we choose $a$ to be an irrational number then $A$ cannot belong to any cone $\CC^2_r$
and, if we choose $a=1/(r+1)$, then $A$ belongs to $\CC^2_{r+1}$ but not to $\CC^2_r$.
\end{proof}
We now show  that the union of the cones $\CC^n_r$  covers the interior of the cone $\cgpsd^n$.

\begin{theorem} \label{thminterior}
We have the inclusions:
\beqn\text{\rm int}(\cgpsd^n) \subseteq \bigcup_{r\ge 1} \CC^n_r \subseteq \cgpsd^n.\eeqn 
\end{theorem}

\begin{proof}
We only need to show the first inclusion.
For this, consider a matrix $A$ in the interior of the cone $\cgpsd^n$ and assume that $A$ does not belong to 
$\bigcup_{r\ge 1} \CC^n_r$.
Then, for each $r\ge 1$, there exists a hyperplane separating $A$ from the (closed convex) cone $\CC^n_r$. That is, there exists a matrix $M_r\in\DD^n_r$ such that $\langle M_r,A\rangle <0$ and $\|M_r\|=1$.
Since all matrices $M_r$ lie in a compact set, the sequence $(M_r)_{r}$ admits a converging subsequence $(M_{r_i})_{i\ge 1}$ which converges to a matrix $M\in \SSS^n$.
By Lemma \ref{lemconv}, we know that the matrix $M$ belongs to the cone $\cgpsd^{n*}$ and thus
$\langle A,M\rangle \ge 0$. On the other hand, as $\langle A,M_{r_i}\rangle <0$ for all $i$, by taking the limit as $i$ tends to infinity, we get that $\langle A,M\rangle \le 0$. Hence we obtain $\langle A,M\rangle=0$, which contradicts the assumption that $A$ lies in the interior of $\cgpsd^n$.
\end{proof}

  It is easy to give an explicit description of the cones $\CC^{n}_r$ for small $r$. 
  For example, $\CC^{n}_1$ is the set of $n \times n$ diagonal  nonnegative matrices and $\CC^{n}_2$ is the convex hull of the matrices $E_{ii}$ and $E_{ii}+E_{ij}+E_{jj}$ (for $i,j\in [n]$), where $E_{ij}$ denote the elementary matrices in $\SSS^n$.

\section{LP lower bounds to the quantum chromatic number}\label{sec:lp-bounds}

In this section we use the polyhedral hierarchy $\CC^{n}_r$ ($r\ge 1$) to show that the parameter $\tchiq(G)$ in (\ref{def:tchiq}) can be written as a linear program.
We recall the definition of $\tchiq(G)$:
\beq \label{eq:tchiq}
\tchiq(G) = \min t \in \mathbb{N} \text{ s.t. } \exists A \in \cl(\cgpsd^{nt}), \, A \in \mathcal{A}^t \text{ and }  L_{G,t}(A) = 0,
\eeq
where the affine space $\mathcal{A}^t$ is defined in (\ref{def:At}) and the map $L_{G,t}$ in (\ref{def:Lt}).
A first natural approach for building a linear relaxation of $\tchiq(G)$ is to replace the cone $\cl(\cgpsd^{nt})$ in the definition of $\tchiq(G)$ by the subcone $\CC^n_r$, leading to the parameter
\beqn\ell_r(G)=\min t \in \N \text{ s.t. } \exists A\in \CC^{nt}_r, \, A \in \mathcal{A}^t \text{ and }  L_{G,t}(A) = 0.\eeqn
As $\CC^{nt}_r\subseteq \cgpsd^{nt}$, we have $\tchiq(G)\le \chi_q(G)\le \ell_r(G)$. 
Moreover the sequence $(\ell_r(G))_r$ is monotone nonincreasing and thus has a limit (it becomes stationary). However it is not clear whether the limit is equal to $\chi_q(G)$.
If one could claim that for $t=\chi_q(G)$ there is a feasible matrix $A$ for the program (\ref{eq:tchiq}) which lies in the interior of $\cgpsd^{nt}$ then,
by Theorem \ref{thminterior},  $A$ would belong to some cone $\CC^{nt}_r$ which would imply equality $\chi_q(G)=\ell_r(G)$.
However, this idea cannot work because, as observed in Lemma \ref{lemborder}, any matrix feasible for (\ref{eq:tchiq}) lies in the boundary of $\cgpsd^{nt}$.
To go around this difficulty, our strategy is to relax the affine constraints in (\ref{eq:tchiq}) so as to allow feasible solutions in the interior of $\cgpsd^{nt}$.

More precisely, given an integer $k\ge 1$, we consider the affine 
space $\mathcal{A}^t_k$  defined by the equations: $\mid \sum_{i,j} A_{ui,vj} - 1 \mid \ \le \frac{1}{k}$ for all $ u,v \in V(G)$. We define the parameter:
\beq \label{eq:lambda_k}
\lambda_k(G) = \min t \text{ s.t. } \exists A \in \cl(\cgpsd^{nt}), \, A \in \mathcal{A}^t_k \text{ and } L_{G,t}(A) \le \frac{1}{k}.
\eeq
In a first step we show that $\lambda_k(G)=\tchiq(G)$ for $k$ large enough.

\begin{lemma}\label{lem:lambda_k}
For any graph $G$, there exists $k_0 \in \N$ such that $\tchiq(G) = \lambda_{k}(G)$ for all $k \ge k_0$. 
\end{lemma}

\begin{proof}
Notice that $\lambda_k(G) \le \tchiq(G)$ holds for every $k \in \N$. Indeed, any matrix solution for $\tchiq(G)$ is also a solution for $\lambda_{k}(G)$.
Moreover, as the sequence $\left(\lambda_k(G)\right)_{k\in \N}$ is a monotone nondecreasing sequence of natural numbers upper bounded by $\tchiq(G)$, there exists a $k_0$ such that $\lambda_k(G) = \lambda_{k_0}(G)$ for all $k \ge k_0$. Let $t = \lambda_{k_0}(G)$.
For all $k \ge k_0$ there exists a matrix $A_k \in \cl(\cgpsd^{nt})$ with $A_k \in \mathcal{A}^t_k$ and $L_{G,t}(A_k)\le \frac{1}{k}$.
Consider the sequence $(A_{k})_{k \ge k_{0}}$, which is bounded as all $A_{k}$ lie in $\mathcal{A}^{t}_{k_{0}}$. Therefore, the sequence has a converging subsequence to, say, $A$ where $A \in \cl(\cgpsd^{nt})$, $A \in \mathcal{A}^{t}$ and $L_{G,t}(A) = 0$.
Hence, $A$ is a feasible solution for $\tchiq(G)$ and $\tchiq(G) \le t = \lambda_{k_{0}}(G) = \lambda_{k}(G)$ for all $k \ge k_{0}$.
\end{proof}

In a second step we show that the new parameter $\lambda_k(G)$ can be computed by a linear program. For this we replace in the definition of $\lambda_k(G)$ the cone $\cl(\cgpsd^{nt})$ by  the polyhedral cone $\CC^{nt}_r$, leading to the following parameter:
\beq \label{eq:lambda_k^r}
\lambda_{k}^{r}(G) = \min t \text{ s.t. } \exists A \in \CC^{nt}_{r}, \, A \in \mathcal{A}^t_k \text{ and } L_{G,t}(A) \le \frac{1}{k}.
\eeq
Notice that this parameter $\lambda_k^r(G)$  can be computed through a linear program since $\CC^{nt}_r$ is polyhedral.
We will show that for any graph $G$ there exist integers $k_{0}$ and  $r_{0}$ such that $\tchiq(G) = \lambda_{k_{0}}^{r_{0}}(G)$. We emphasize that this is an existential result: we do not know for which integers $k_{0}$ and $r_{0}$ such a convergence happens.
One of the ingredients to prove the result is to show the existence of a matrix in the interior of $\cgpsd$ satisfying certain constraints. To this end, we will use the matrix $Z = I + J \in \SSS^{nt}$ where $I$ and $J$ are, respectively, the identity and the all-ones matrices.

\begin{lemma}\label{lemZIJ}
The matrix $Z=I+J \in \SSS^{nt}$ lies in the interior of $\cgpsd$. 
Moreover, we have that $\sum_{i,j \in [t]} Z_{ui,uj} = t^{2}+t$ for all $u \in V(G)$, $\sum_{i,j \in [t]} Z_{ui,vj} = t^{2}$ for all $u \neq v \in V(G)$ and $L_{G,t}(Z) = nt^{2} - nt + mt$, where $m$ is the number of edges of the graph $G$.
\end{lemma}

\begin{proof}
We only show that $I+J$ lies in the interior of $\cgpsd^{nt}$, the other claims are direct verification. 
Assume that there exists a matrix $M\in \cgpsd^{nt*}$ such that $\langle M,I+J\rangle=0$; we show that $M=0$.
Indeed, as both $I$ and $J$ lie in $\cgpsd^{nt}$ we get that $\Tr(M)=0$ and $\langle J,M\rangle =0$.
Observe that since $M$ is copositive with zero diagonal entries, all its entries must be nonnegative. Combining with $\langle J,X\rangle=0$,
 we deduce that $M$ is identically zero.
\end{proof}

\begin{theorem}\label{theotchiq}
For any  graph $G$ there exist $k_{0}$ and $r_{0} \in \mathbb{N}$ such that $\tchiq(G) = \lambda_{k}^{r}(G)$ for all $k\ge k_{0}$ and all $r \ge r_{0}$.  Moreover $\lambda_{k_{0}}^{r_{0}}(G)$, and thus $\tchiq(G)$, can be computed via a linear program.
\end{theorem}

\begin{proof}
From Lemma~\ref{lem:lambda_k} we know that there exists $k_{0} \in \mathbb{N}$ such that $\lambda_{k}(G) = \tchiq(G)$ for all $k \ge k_{0}$. 
In view of this, we just need to show that for this $k_{0}$ there exists an integer $r_{0} \in \mathbb{N}$ for which $\lambda_{k_{0}}^{r_{0}}(G) = \lambda_{k_{0}}(G)$. Let $t = \lambda_{k_{0}}(G) = \tchiq(G)$.

By the definitions~(\ref{eq:lambda_k}) and~(\ref{eq:lambda_k^r}) and the inclusion relationship between the cones $\CC_{r}^{nt}$, 
we have that the sequence of natural numbers $(\lambda_{k_{0}}^{r})_{r \in \mathbb{N}}$ is nonincreasing and it is lower bounded 
by $\lambda_{k_{0}}(G)$. Hence, there exists a natural number $r_{0}$ such that $\lambda_{k_{0}}^{r}(G) = \lambda_{k_{0}}^{r_{0}}(G)$ for all $r \ge r_{0}$.
 We are left to prove that $\lambda_{k_{0}}^{r_{0}}(G) \le \lambda_{k_{0}}(G)=t.$

To this end, we show that there exists a matrix $Y_{k_{0}} \in \text{\rm int}(\cgpsd)$ with $Y_{k_{0}} \in \mathcal{A}_{k_{0}}^{t}$ and $L_{G,t}(Y_{k_{0}}) \le \frac{1}{k_{0}}$.
This will suffice since then 
by Theorem~\ref{thminterior}, $Y_{k_{0}} \in \CC_{r_{0}}^{nt}$ for some $r_0$. Therefore, $Y_{k_{0}}$ satisfies the conditions in program~(\ref{eq:lambda_k^r}) and thus $\lambda_{k_{0}}^{r_{0}}(G) \le t = \lambda_{k_{0}}(G)$. 
To show the existence of such a matrix $Y_{k_{0}}$, let $A\in \cl(\cgpsd)$ be a feasible solution of the program~(\ref{eq:tchiq}) defining  $\tchiq(G) = t$ and consider the matrix 
$Z = I + J $ which belongs to $\text{\rm int}(\cgpsd)$ (by Lemma \ref{lemZIJ}).
Then, any convex combination  $Z_\varepsilon=(1- \varepsilon )A+ \varepsilon Z$ (for $0<\varepsilon <1$)  lies in the interior of $\cgpsd$. 
If we can tune $\varepsilon$ so that the new matrix $Z_\varepsilon$ satisfies the conditions in program~(\ref{eq:lambda_k^r}), then we can choose $Y_{k_0}=Z_\varepsilon$ and we are done.
We claim that selecting $\varepsilon := \min\{ \frac{1}{k_{0}(t^{2}+t-1)}, \frac{1}{k_{0}(nt^{2}-nt+2mt)} \}$ will do the trick. 
Indeed, for this choice of $\varepsilon$ we have $Z_\varepsilon \in \text{\rm int}(\cgpsd)$ and $L_{G,t}(Z_\varepsilon) = \varepsilon L_{G,t}(Z) \le \frac{1}{k_{0}}$ (use Lemma \ref{lemZIJ}).
Moreover, $Z_\varepsilon \in \mathcal{A}_{k_{0}}^{t}$ since for all $u,v \in V(G)$ the following holds
{\small 
\beqn
\big| \sum_{i,j \in [t]} Y_{k_{0}}(ui,vj) - 1 \big| 
= \big| (1-\varepsilon) + \varepsilon \sum_{i,j \in [t]} Z_{ui,vj} -1 \big|
\le \big| -\varepsilon + \varepsilon \sum_{i,j \in [t]} Z_{ui,uj}  \big|
 =  \big|  \varepsilon (t^{2}+t - 1) \big| \le \frac{1}{k_{0}}.
\eeqn}
Summarizing, from Lemma~\ref{lem:lambda_k} we know that there exists an integer $k_{0}\in\N$ such that $\lambda_{k_{0}}(G) = \tchiq(G)$ and we just proved that for this $k_{0}$ there exists an integer $r_{0}\in\N$ with the property that $\lambda_{k_{0}}^{r_{0}}(G) = \lambda_{k_{0}}(G) = \tchiq(G)$.
\end{proof}

The same result holds for the parameter $\widetilde{\chi}_{qa}(G)$ introduced in (\ref{def:tchiqa}). For clarity we repeat its definition  in the following form:
\beqn\widetilde{\chi}_{qa}(G) =\min t\in\N \text{ s.t. } \exists A \in \cl(\cgpsd^{2nt}), A\in \mathcal B^t \text{ with } \mathcal{L}_{G,t}(\pi(A))=0.\eeqn
Note the analogy with the definition (\ref{eq:tchiq}) of $\widetilde \chi_q(G)$. The only difference is that we now work with matrices $A$ of size $2nt$ (instead of $nt$) lying in the affine space ${\mathcal B}^t$ (instead of $\mathcal A^t$) and satisfying $\mathcal L_{G,t}(\pi(A))=0$ (instead of $L_{G,t}(A)=0$).
In analogy to the parameter $\lambda_k(G)$ we can define the parameter $\Lambda_k(G)$ by doing these replacements and defining the relaxed affine space $\mathcal{B}^{t}_{k}$ in the same way as $\mathcal{A}^{t}_{k}$ was defined from $\mathcal{A}^{t}$. Then the analog of Lemma \ref{lem:lambda_k} holds: there exists an integer $k_0$ such that $\widetilde{\chi}_{qa}(G)=\Lambda_k(G)$ for all $k\ge k_0$.
Next, replacing the cone $\cl(\cgpsd^{2nt})$ by $\CC^{2nt}_r$, we get the following parameter $\Lambda^r_k(G)$ (the analog of $\lambda^k_r(G)$):
\beqn\Lambda^r_k(G)=\min t \in \mathbb{N} \text{ s.t. } 
A \in \CC^{2nt}_r, A \in \mathcal B^t_{k} \text{ with }  {\mathcal L}_{G,t}(\pi(A))\le \frac{1}{k}.
\eeqn
The  analog of Theorem \ref{theotchiq} holds, whose proof is along the same lines and thus omitted.
\begin{theorem}
For any graph G, there exist $k_0$ and $r_0 \in \N$  such that $\widetilde \chi_{qa}(G) = \Lambda^r_k(G)$  for all $k\ge k_0$ and $r\ge r_0$.
Hence the parameter $\widetilde{\chi}_{qa}(G)$ can be computed by a linear program.
\end{theorem}

\section{The closure of \texorpdfstring{$\cgpsd$}{CS}}\label{sec:closure}

In the Introduction we have mentioned that if the completely positive semidefinite cone would be closed, then the set of quantum correlations would be closed as well (see also \cite{MR:2014}). 
Although we still do not know whether $\cgpsd$ is closed, in this section we make a small progress by giving a new description of the closure of $\cgpsd$,  using the tracial ultraproduct of matrix algebras $\R^{k\times k}$. More precisely, the closure $\cl(\cgpsd)$ consists of the symmetric matrices having a Gram representation by positive semidefinite operators which belong to the above mentioned tracial ultraproduct. This ultraproduct will be an algebra of bounded operators on an infinite dimensional Hilbert space. 

\smallskip
A connection between $\cl(\cgpsd)$ and the Gram matrices of operators on infinite dimensional Hilbert spaces has already been made by two of the authors in \cite{LP13}. Namely, let $\mathcal S^\N$ denote the vector space of all infinite symmetric matrices $X=(X_{ij})$ indexed by $\N$ with finite norm $\sum_{i,j\geq 1}X_{ij}^2<\infty$, equipped with the inner product $\langle X, Y\rangle=\sum_{i,j\geq 1} X_{ij}Y_{ij}$. Using this notation, we let $\mathcal{CS}^n_{\infty+}$ denote the convex cone of matrices $A\in\mathcal S^n$ having a Gram representation by positive semidefinite matrices in $S^\N$. Then it is shown in \cite{LP13} that $\cgpsd\subseteq \mathcal{CS}_{\infty+}\subseteq\cl(\mathcal{CS}_{\infty+})=\cl(\cgpsd)$ holds. In particular, the closure of $\cgpsd$ a priori contains matrices having a Gram representation by infinite dimensional matrices. 

\smallskip
Tracial ultraproducts of matrix algebras, or more generally of finite von Neumann algebras, are an adapted version of classical ultraproducts from model theory. Since the methods used might be not familiar to the reader, we recap the construction of tracial ultraproducts. 
Then we introduce the new cone $\csu$ and show that it is equal to the closure of $\cgpsd$.
Finally, we present a possibly larger cone $\csvn$, containing $\cgpsd$, which can be seen as an infinite dimensional analog of the completely positive semidefinite cone. This cone turns out to be closed. Furthermore, $\csvn$ would be equal to $\cl(\cgpsd)$ if the embedding problem of Connes had an affirmative answer.
More details about the algebras involved in the general case as well as on the embedding problem of Connes are given in Section \ref{sec:connes}. 

\subsection{Tracial ultraproducts}
The construction of tracial ultraproducts is a standard technique in von Neumann algebras, see, e.g., the appendix of \cite{BObook}.
Classically one considers complex Hilbert spaces but the construction works similarly over real Hilbert spaces. Alternatively one can use the complex construction and `realify' the resulting algebra afterwards, see for instance \cite{ARUbook,Libook}.
Ultraproducts are constructions with respect to an ultrafilter. We will only consider ultrafilters on $\N$.   
Throughout $\pow(\N)$ is the collection of all subsets of $\N$.

\begin{definition}
An {\rm ultrafilter} on the set $\N$ is a subset  $\cU\subseteq\pow(\N)$ satisfying the conditions:
\begin{enumerate}[(a)]\itemsep0pt
\item $\emptyset\notin \cU$, \label{empty}
\item if $A\subseteq B\subseteq \N$ and $A\in\cU$ then $B\in\cU$,
\item if $A,B\in\cU$ then $A\cap B\in\cU$,\label{disjointcond}
\item for every $A\in\pow(\N)$ either $A\in\cU$ or $\N\setminus A\in\cU$.\label{ultracond}
\end{enumerate} 
\end{definition}

In particular, any two elements in $\cU$ need to have non-empty intersection (from \eqref{empty} and \eqref{disjointcond}), which allows
only two kinds of ultrafilters: Either every element of $\cU$ contains a common element $n_0\in \N$ or $\cU$ contains the cofinite sets of $\N$. We are only interested in the second kind of ultrafilters, which are  called  \textit{free ultrafilters}. 
For a given free ultrafilter $\cU$ on $\N$ we can define the \textit{ultralimit} $\lim_\cU a_k$ of
a bounded sequence $(a_k)_{k\in\N}$ of real numbers as follows:
\begin{align}\label{eq:ulimit}
\lim_\cU a_k=a
\text{ if } \{k\in\N: |a_k-a|<\varepsilon\}\in\cU \ \text{ for all } \varepsilon>0.
\end{align} 

Let us  have a look at ultralimits in a less formal way. If we have a non-free ultrafilter, i.e., $\cU=\{A\in\pow(\N): k_0\in A\}$ for some $k_0\in\N$, then $
\lim_\cU a_k=a_{k_0}$ for any sequence $(a_k)_{k\in\N}\subseteq\R$. The case of a free ultrafilter is more interesting. Then the ultralimit of a bounded sequence $(a_k)_{k\in\N}$ will be 
one of its accumulation points. For example, the sequence given by $a_k:=(-1)^k$ for all $k\in\N$ has two accumulation points, and both can be attained as an ultralimit depending on the choice of the ultrafilter $\cU$. In fact, considering the set $2\N$ of even numbers, we get by conditions \eqref{disjointcond} and \eqref{ultracond} that any ultrafilter contains either $2\N$ or its complement (the odd numbers $2\N+1$) but not both. Hence there is an ultrafilter $\cU$ (containing $2\N$) with $\lim_\cU a_k=1$ and an ultrafilter $\cU'$ (containing $2\N+1$) with $\lim_{\cU'}a_k=-1$.  

\begin{rem}\label{rem:ulimit}
Any bounded sequence of real numbers has an ultralimit and this is unique for fixed $\cU$. In particular, if $\lim_{n\to\infty} a_k=a$ then $\lim_\cU a_k=a$ for any free ultrafilter $\cU$ on $\N$.
\end{rem}

We can use ultralimits to construct 
 the tracial ultraproduct of a sequence $(\R^{d_k\times d_k})_{k\in\N}$ of matrix algebras for $d_k\in\N$. To simplify notation we let $\cM_k=\R^{k\times k}$ denote the matrix algebra of all $k\times k$ matrices and we consider the full sequence $(\cM_k)_{k\in\N}$, but the same construction would work for the sequence $(\cM_{d_k})_{k\in\N}$.
Here we assume that each $\cM_k$ is endowed with the normalized trace $\nt_k=\frac1k\Tr$ (if the dimension $k$ is clear we might simply write $\nt$) and the corresponding inner product, so that $\|I\|_2=\nt(I)=1$ for the identity matrix.
For $T\in \cM_k$, $\|T\|$ denotes its operator norm and $\|T\|_2$ its $L_2$-norm, that satisfy  $\|ST\|_2\le \|S\| \|T\|_2$ for $S,T\in\cM_k$.
Define the $C^*$-algebra 
\beqn\ell^\infty(\N,(\cM_k)_k)
:=\{(T_k)_{k\in\N}\in\prod_{k\in\N}\cM_k: \sup_{k\in\N}\|T_k\|<\infty\}.\eeqn
Every free ultrafilter $\cU$ on $\N$ defines a two-sided ideal
\beqn\mathcal I_\cU:=\{(T_k)_{k\in\N}\in\ell^\infty(\N,(\cM_k)_k): \lim_{\cU}\|T_k\|_2=0\},\eeqn  
which is well-defined since sequences in $\ell^\infty(\N,(\cM_k)_k)$ are also bounded in the Hilbert-Schmidt norm.
The ideal $\mathcal I_\cU$ is a maximal ideal and therefore it is closed with respect to the operator  norm. 
The quotient algebra 
\beqn\cM_\cU:=\ell^\infty(\N,(\cM_k)_k)/\mathcal I_\cU\eeqn is called the
\textit{tracial ultraproduct} of $(\cM_k)_k$ along $\cU$. 
Using the Cauchy-Schwarz inequality it is easy to show that the map
\begin{align*}
\tau_{\cU}:\; \cM_\cU \to\R,\quad
(T_k)_{k\in\N}+\mathcal I_\cU \mapsto\lim_{\cU}\nt_k(T_k) 
\end{align*}
defines a tracial state (or trace) on $\cM_\cU$, i.e., $\tau_{\cU}$ is a normalized positive linear map satisfying $\tau_\cU(T^*T)=\tau_\cU(TT^*)$ for any $T\in\cM_\cU$.
In fact, $\cM_\cU$ is a finite von Neumann algebra of type II$_1$ (see below for definitions). In particular, $\cM_\cU$ is a subalgebra of bounded operators on an infinite dimensional Hilbert space.

\subsection{Von Neumann algebras and Connes' embedding problem} \label{sec:connes}
We give a short overview of what is needed for our purpose; for details we refer to the book \cite{TakII}. 

A \textit{von Neumann algebra} $\cN$ is a unital $*$-subalgebra
of the $*$-algebra $\cB(\cH)$ of bounded operators on a Hilbert 
space $\cH$ that is closed in the weak operator topology.
The \textit{weak operator topology} is the weakest topology on $\cB(\cH)$ such that the functional $\cB(\cH)\to\C$ which maps $T\mapsto \langle Tx,y\rangle$ is continuous for any 
$x,y\in \cH$. 
In other words, 
a sequence $(T_i)_i \in \cB(\cH)$ converges to $T\in \cB(\cH)$ in the weak $*$-topology 
 if for all $x,y\in \cH$ the sequence $(\langle T_ix,y \rangle)_i$ converges to $\langle Tx,y \rangle$. 

A \textit{factor} is a von Neumann algebra with trivial center. 
Every von Neumann algebra on a separable Hilbert space 
is isomorphic to a direct integral of factors, which is the appropriate analog of matrix block decomposition.

A factor $\cF$ is \emph{finite} if it possesses a normal,
faithful, tracial state $\tau:\mathcal F\to\C$. In particular, we can always assume that $\tau(I)=1$. This tracial state
$\tau$ is unique and gives rise to the Hilbert-Schmidt norm on $\mathcal F$
given by $\|T\|_2^2:=\tau(T^*T)$ for $T\in\mathcal F$.
A von Neumannn algebra is \textit{finite} if it decomposes into finite factors. Every finite von Neumann algebra comes with a trace, which might not be unique.  

Von Neumann algebras can be classified into two types depending on  the behavior of their projections (i.e., the elements $P\in \cN$ satisfying $P=P^*=P^2$). If for a given finite factor $\cF$ with trace $\tau$ the range of $\tau$ over all projections $P\in\mathcal F$ is discrete, 
then $\mathcal F$ is of \textit{type} I. A von Neumann algebra is of type I if it consists only of type I factors. Any finite type I von Neumann algebra is isomorphic to a matrix algebra over $\C$. The only other possibility for a finite factor is that $\tau$ maps projections (surjectively) onto $[0,1]$. Those are II$_1$ factors, and a von Neumann algebra is of type II$_1$ if it is finite and contains at least one II$_1$ factor.  

\medskip
Connes' embedding problem asks to which extent II$_1$ factors 
are close to matrix algebras. Murray and von Neumann showed that
there is a unique II$_1$ factor $\mathcal R$ which contains an ascending
sequence of finite-dimensional von Neumann subalgebras, i.e. matrix algebras, with dense union. This factor $\mathcal R$ is called the
\emph{hyperfinite {\rm II}$_1$ factor}. There are several constructions 
of $\mathcal R$, e.g., as infinite tensor product $\overline\bigotimes_{n\in\N}M_2(\C)$ 
of the von Neumann algebras $M_2(\C)$, which is the weak closure of the 
algebraic tensor product $\bigotimes_{n\in\N}M_2(\C)$. In fact, any infinite countable sequence of matrix algebras will do. 

Connes 
conjectured that all separable II$_1$ factors embed (in a trace-preserving way) into an ultrapower
$\cR^\cU$ of the hyperfinite II$_1$ factor $\cR$, where the ultrapower $\cR^\cU$ is just a short-hand notation for the ultraproduct 
$\ell^\infty(\N,(\cR)_k)/\mathcal I_\cU$. 
Since $\cR$ contains ascending sequences of matrix algebras with dense union, any matrix algebra $\cM_k$ embeds into $\cR$. One can extend these embeddings of $\cM_k$ into $\cR$ to an embedding of the tracial ultraproduct $\cM_\cU$ into $\cR^\cU$ (using a more general construction of ultralimits), hence the finite von Neumann algebra $M_\cU$ satisfies Connes' embedding conjecture.

This conjecture is equivalent to a huge variety of other important conjectures in, e.g.,  operator theory, noncommutative real algebraic geometry and quantum information theory.
In particular, as we already mentioned in the introduction, it is equivalent to deciding whether $\cl(\mathcal Q)=\mathcal Q_c$ holds. 

\smallskip
For our description of $\cl(\cgpsd)$, we will use the following result on finite von Neumann algebras which embed into $\cR^\cU$. This proposition is applied to the algebra $\cM_\cU$. 
The claim is that tracial moments of an embeddable finite factor can be approximated up to arbitrary precision by matricial tracial moments. This is stated more formally in the next proposition, for a proof see e.g. \cite{CD:2008}.

\begin{proposition}\label{prop:micro}
Let $(\cF,\tau)$ be a II$_1$ factor which embeds into $\cR^\cU$ for some
free ultrafilter $\cU$. Then $\cF$ has matricial microstates, i.e., for any $n\in\N$ 
and given self-adjoint $T_1,\dots,T_n\in\cF$ the following holds:
for every $k\in\N$ and   $\ep>0$ there exists  $d\in\N$ and $B_1,\dots,B_n\in\mathcal S^d$ such that 
\begin{align*}
|\tau(T_{i_1}\cdots T_{i_t})-\nt(B_{i_1}\dots B_{i_t}))|<\ep \ \text{ for all } i_1,\dots,i_t\in[n], t\leq k.
\end{align*}
\end{proposition}

\subsection{Ultraproduct description of \texorpdfstring{$\cl(\cgpsd)$}{clCS}}

We are now ready to define the new cone $\csu$ which will turn out to be equal to the closure of $\cgpsd$. For this, we fix a free ultrafilter $\cU$ on $\N$ and consider the tracial ultraproduct
$\cM_\cU=\ell^\infty(\N,(\cM_k)_k)/\mathcal I_\cU$ 
where again $\cM_k$ denotes the
full matrix algebra $\R^{k\times k}$ for any $k\in\N$. Using this we define
\beqn\csu:=\{A\in\cpsd: A=(\tau_\cU(X_iX_j)) \text{ for some  positive semidefinite } X_1,\dots,X_n\in {\cM_\cU}\}.\eeqn 
We note that the trace $\tau_\cU$ is normalized (i.e.,  $\tau_\cU(I)=1$) whereas we used the (not normalized) trace $\Tr$ in the definition of $\cgpsd$. 
However, both descriptions agree up to rescaling of the $X_i$'s.

To show that the closure of $\cgpsd$ is a subset of $\csu$  we will consider a sequence of matrices 
$A^{(k)}\in \cgpsd^n$ converging to some $A\in\mathcal S^n$, i.e., $\lim_{k\to \infty}A^{(k)}_{ij}=A_{ij}$ for all $i,j\in [n]$.
A priori, for each $k$, there exist an integer $d_k$ and matrices $X^{(k)}_1,\ldots,X^{(k)}_n\in \mathcal S^{d_k}_+$ such that 
$A^{(k)}=(\nt(X^{(k)}_iX^{(k)}_j))$. The next lemma says that without loss of generality we can assume $d_k=k$ for all $k\in\N$.

\begin{lemma}\label{lem:order}
If $(X_k)_k,(Y_k)_k\in \prod_{k\in\N} \cpsd^{d_k}$ are such that the sequence $(\nt_{d_k}(X_kY_k))_{k\in\N}$ 
converges to some $a\in\R$, then there exist $(X'_k)_k,(Y'_k)_k\in \prod_{k\in\N} \cpsd^{k}$ with $\nt_k(X'_kY'_k)\to a$ as $k\to\infty$.
\end{lemma}
\begin{proof}
By possibly reordering the indices we can assume that the sequence $(d_k)_{k\in\N}$ is monotonically nondecreasing. 
First, we modify the sequence $(X_k)_k$ in such a way that $d_k\leq k$ holds for all $k\in\N$. For this, if there is some $k\in\N$ with $d_k>k$ 
we repeat the preceding element $X_{k-1}$ exactly $d_k-k$ times before the element $X_k$. For instance, if $X_1\in\R_+$ and $X_2\in\cpsd^3$ (i.e., $d_1=1$ and $d_2=3$), we replace the sequence $(X_1,X_2,X_3,\dots)$ by $(X_1,X_1,X_2,X_3,\dots)$. Then the position of $X_k$ is shifted by $d_k-k$ to $k+d_k-k=d_k$. If $k=1$ we simply add $d_1-1$  zero matrices  before $X_1$.
We do the same with the sequence $(Y_{k})_k$. Then the new sequence of inner products is obtained from the original sequence $(\nt_{d_k}(X_kY_k))_{k\in\N}$ by replacing each $\nt_{d_k}(X_kY_k)$ by $d_k-k+1$ copies of it if $d_k>k$, and thus still converges to the limit $a$.

Thus we can now assume that $d_k\leq k$ for all $k\in\N$. We set 
$X_k':=\sqrt{\frac{k}{d_k}}(X_k\oplus 0_{k-d_k})\in\cpsd^k$ and 
$Y_k':=\sqrt{\frac{k}{d_k}}(Y_k\oplus 0_{k-d_k})\in\cpsd^k$ for every $k\in\N$.
Then we have
\begin{align*}
\nt_{k}(X_k' Y_k')=\frac1k \Tr(X'_k Y'_k)=
\frac1k\frac{k}{d_k}\Tr(X_k Y_k)=\nt_{d_k}(X_k Y_k)
\end{align*} 
for every $k\in\N$. Hence the final sequence $(\nt_{k}(X'_k Y'_k))_{k\in\N}$ still converges to $a$. 
\end{proof}
We proceed by showing that the closure of $\cgpsd$ is a subset of $\csu$. The main ingredient will be Remark \ref{rem:ulimit}
together with the result of Lemma \ref{lem:order}.
\begin{lemma}\label{lem:contain}
For any free ultrafilter $\cU$ on $\N$, we have
$\cl(\cgpsd)\subseteq \csu$.
\end{lemma}
\begin{proof}
Let $A\in\cl(\cgpsd)$ be given. Then there is a sequence of matrices $A^{(k)}\in \cgpsd$ converging to $A$:
$\lim_{k\to \infty}A^{(k)}_{ij}=A_{ij}$ for all $i,j\in [n]$. 
For each $k$, there exist positive semidefinite matrices $X^{(k)}_1,\ldots,X^{(k)}_n$ such that 
$A^{(k)}=(\nt(X^{(k)}_iX^{(k)}_j))$. By Lemma \ref{lem:order} we can assume that $X^{(k)}_1,\ldots,X^{(k)}_n\in \mathcal S^{k}_+$. 
As the matrices $A^{(k)}$ are bounded the matrices $X^{(k)}_i$ are bounded as well.
Hence the sequence $(X^{(k)}_i)_k$ belongs to $\ell^\infty(\N,(\cM_{k}))$ and we can consider its image $X_i$ in the tracial ultrapower $\cM_\cU$.
By the theorem of \L{}os the operators $X_i$ are positive semidefinite since all $X^{(k)}_i$ are positive semidefinite.
It suffices now to show that $A=(\tau_\cU(X_iX_j))$ since then we can conclude that $A\in \csu$.
For this observe that, by the definition of $\tau_\cU$, we have:
$\tau_\cU(X_iX_j)=\lim_\cU \nt(X^{(k)}_iX^{(k)}_j) =\lim_\cU A^{(k)}_{ij}$.
On the other hand, as the sequence $(A^{(k)}_{ij})_k$ converges to $A_{ij}$, in view of Remark \ref{rem:ulimit}, we have that 
$\lim_{\cU}A^{(k)}_{ij}=A_{ij}.$ This concludes the proof.
\end{proof}

Since Connes' embedding conjecture holds true for the tracial ultraproduct $M_\cU$, i.e., $\cM_\cU$ embeds into the ultrapower $\cR^\cU$ of the hyperfinite II$_1$ factor $\cR$, we get by Proposition \ref{prop:micro} that $M_\cU$ has matricial microstates. This will be the key ingredient to show the equality between $\csu$ and $\cl(\cgpsd)$.

\begin{theorem} \label{prop:equality}
For any free ultrafilter $\cU$ on $\N$
$\cl(\cgpsd)=\csu$ holds.
\end{theorem}

\begin{proof}
In view of Lemma \ref{lem:contain} we only have to show the inclusion $\csu \subseteq \cl (\cgpsd)$.
Let $A\in \csu$. By assumption, $A=(\tau_\cU(X_iX_j))$ for some positive semidefinite operators $X_1,\ldots,X_n\in \cM_\cU$.
As the operators $X_i$ are positive semidefinite, there exist operators $Y_i$ such that $X_i=Y_i^2$ for $i\in [n]$. 
Since $M_\cU$ embeds into $\cR^\cU$  we can apply Proposition \ref{prop:micro} and conclude that $M_\cU$ has matricial microstates for the operators $Y_1,\ldots,Y_n$. In particular, for every $k\in\N$, there exist $d_k\in\N$ and symmetric  matrices $B_1^{(k)},\ldots, B^{(k)}_n\in \mathcal S^{d_k}$ 
such that $|\tau_\cU(Y_i^2Y_j^2)-\nt((B_i^{(k)})^2(B_j^{(k)})^2)|<1/k$. Hence the operators $X_i^{(k)}=(B_i^{(k)})^2$ belong to
$\mathcal S^{d_k}_+$ and satisfy 
\begin{align}\label{eq:approx}
|\tau_\cU(X_iX_j)-\nt({X_i^{(k)}}{X_j^{(k)}})|<\frac1k \ \text{ for all } i,j\in [n].
\end{align} 
For each $k$, the matrix $A^{(k)}:=(\nt(X^{(k)}_iX^{(k)}_j))$ belongs to the cone $\cgpsd$.
Moreover it follows from (\ref{eq:approx}) that the sequence $(A^{(k)})_k$ converges to the matrix $A$.
This shows that $A$ belongs to the closure of $\cgpsd$, which concludes the proof.
\end{proof}

We would like to conclude with a possible other description of the closure of $\cgpsd$ in
the case that Connes' embedding conjecture turns out to be true. 

As mentioned at the beginning of the section, the closure of $\cgpsd$  contains the cone $\mathcal{CS}_{\infty+}$, i.e., it contains symmetric matrices which have a Gram representation by some class of positive semidefinite infinite dimensional matrices. Also the given description of $\cl(\cgpsd)$ as $\csu$ involves Gram representations by operators on an infinite dimensional Hilbert space. In regard to the relativistic model of quantum correlations where one allows \emph{all} (possibly infinite dimensional) Hilbert spaces one might ask for the most general infinite dimensional version of $\cgpsd$. Since 
one is restricted to operators for which one can define an inner product (or a trace), a decent candidate for the infinite dimensional analog of $\cgpsd$ is 
\beqn\csvn:=\{A\in\cpsd: A=(\tau_{\mathcal N}(X_iX_j)) \text{ {\small for a finite vN algebra }} \mathcal N \text{ {\small and psd }} X_1,\ldots,X_n\in \mathcal N\},\eeqn
where we allow \emph{any} finite von Neumann algebra $\mathcal N$ (with trace $\tau_{\mathcal N})$. Obviously we have the chain of inclusions $\cgpsd\subseteq\csu\subseteq\csvn$.

Moreover, using the general theory of tracial ultraproducts of von Neumann algebras (instead of just matrix algebras),
 one can show with a similar line of reasoning as in Lemma \ref{lem:contain} that $\csvn$ is closed. Indeed, take a 
sequence of matrices $A^{(k)}\in \csvn^n$ converging to some $A\in\mathcal S^n$. Then
$\lim_{k\to \infty}A^{(k)}_{ij}=A_{ij}$ for all $i,j\in [n]$ and 
for each $k$, there exist a finite von Neumann algebra $\mathcal N_k$ with trace $\tau_k$
and bounded positive operators $X^{(k)}_1,\ldots,X^{(k)}_n\in \mathcal N_k$ such that 
$A^{(k)}=(\tau_k(X^{(k)}_iX^{(k)}_j))$. Fixing a free ultrafilter $\cU$ one can conclude that the images $X_i$ of the sequences $(X^{(k)}_i)_k$ in the tracial ultraproduct $\mathcal N_\cU=\ell^\infty(\N,(\mathcal N_{k})_k)/\mathcal I_\cU$ of the corresponding finite von Neumann algebras provide a Gram representation for $A$ in the von Neumann algebra $\mathcal N_\cU$. Hence the following statement holds.

\begin{theorem} \label{prop:vncone}
$\csvn$ is a closed cone.
\end{theorem}

In this context, we would like to mention a result in \cite{FW:2014} showing that 
 $\csvn^n\subsetneq\cpsd^n\cap\R_+^{n\times n}$ for any $n \ge 5$. Summarizing we have the inclusions:
 \beqn\cl(\cgpsd^n)= \csu^{n}\subseteq \csvn^n\subseteq \cpsd^n\cap \R^{n\times n}_+.\eeqn
Finally, if Connes' embedding conjecture is true then the argument of Proposition 
\ref{prop:equality} could be used for any finite von Neumann algebra and thus this 
would imply that $\cl(\cgpsd)=\csvn$. 

\subparagraph*{Acknowledgments}

S.~B. and T.~P. were funded by the SIQS European project.


\end{document}